\documentclass[12pt, reqno]{amsart}
\usepackage{amsfonts,amssymb,latexsym,amsmath, amsxtra}
\usepackage{verbatim}
\usepackage[all]{xy}
\usepackage[dvips]{graphics}

\usepackage{amsthm,amssymb,amstext,amscd,amsfonts,amsbsy,amsxtra,latexsym}
\usepackage{float}
\usepackage[latin1]{inputenc}

\usepackage[OT2,T1]{fontenc}
\DeclareSymbolFont{cyrletters}{OT2}{wncyr}{m}{n}
\DeclareMathSymbol{\Sha}{\mathalpha}{cyrletters}{"58}

\theoremstyle{plain}
\newtheorem{theorem}{Theorem}[section]

\newtheorem{conjecture}[theorem]{Conjecture}

\newtheorem{lemma}[theorem]{Lemma}
\newtheorem{proposition}[theorem]{Proposition}

\theoremstyle{definition}

\theoremstyle{remark}

\newtheorem*{remark}{Remark}

\numberwithin{equation}{section}

\newcommand{\R}{\mathbb R}
\newcommand{\N}{\mathbb N}

\newcommand{\C}{\mathbb C}

\newcommand{\Q}{{\mathbb Q}}

\renewcommand\Re{\operatorname{Re}}

\def\P{ {\bf P}}

\def\({\left(}
\def\){\right)}
\def\<{\left<}
\def\>{\right>}

\newcommand{\abs}[1]{\left|#1\right|}

\def\cI{\mathcal{I}}

\begin{document}
\title[the Andrews-Zagier asymptotics]{ 
On the Andrews-Zagier asymptotics for partitions without sequences}
\author{Kathrin Bringmann}
\address{Mathematical Institute, University of Cologne, Weyertal 86-90, 50931 Cologne, Germany}
\email{kbringma@math.uni-koeln.de}
\author{Daniel Parry}
\address{Department of Mathematcs, Drexel University, Philadelphia, PA 19104}
\email{dtp29@drexel.edu}
\author{Robert Rhoades}
\address{Center for Communications Research, 805 Bunn Dr., Princeton, NJ 08540,U.
S.A. }
\email{rob.rhoades@gmail.com}
\date{\today}
\maketitle

\section{Introduction and statement of results}
Holroyd, Liggett, and Romik \cite{HLR}  introduced the following probability models:
Let $0<s<1$ and 
$\mathcal{C}_1, \mathcal{C}_2, \cdots$ be independent events with probabilities
$$\P_s(\mathcal{C}_n):= 1-e^{-ns}$$
under a certain probability measure $\P_s$.
Let $A_k$ be the event
$$A_k: = \bigcap_{j=1}^\infty \left( \mathcal{C}_j \cup \mathcal{C}_{j+1} \cup \cdots \cup \mathcal{C}_{j+k-1}\right)$$
that there is no sequence of $k$ consecutive $\mathcal{C}_j$ that do not occur.
With $q:=e^{-s}$ set 
\begin{equation*}
g_k(q):= \P_s(A_k).\end{equation*}
To solve a problem in bootstrap percolation, 
Holroyd, Liggett, and Romik established an asymptotic for $\log(g_k(e^{-s}))$. 

Interestingly, the above described probability model also appears in the study of
integer partitions \cite{AEPR, HLR}. In particular, 
\[
G_k(q)=g_k(q)  \prod_{n=1}^\infty \frac{1}{1-q^n} 
\]
is the generating function for the number of integer partitions
without $k$ consecutive part sizes.  Partitions without 2 consecutive parts
have a celebrated history in relation to the famous Rogers-Ramanujan 
identities.  See MacMahon's book \cite{macmahon} 
or the works of Andrews \cite{And67a, And67b, AndPNAS} 
for more about such partitions. 

Andrews \cite{AndPNAS} found that the key to understanding the
function when $k=2$ lies in Ramanujan's mock theta function
$$\chi(q): = 1 + \sum_{n=1}^\infty \frac{q^{n^2}}{\prod_{j=1}^n \left(1-q^j + q^{2j}\right)}.$$ Namely, he proved that
$$g_2(q) = \chi(q)  \prod_{n=1}^\infty \frac{\left(1+q^{3n}\right)}{\left(1-q^n\right) \left(1-q^{2n}\right)}.$$
From this, an asymptotic expansion for $g_2(e^{-s})$ may obtained, 
see \cite{BM1}.
  Using additional $q$-series identities when $k>2$, Andrews made the following conjecture.
\begin{conjecture}[Andrews \cite{AndPNAS}]\label{conj:main} For each $k\ge 2$, there exists a positive constant $C_k$ such that, as $s\to 0$,
$$g_k\left(e^{-s}\right) \sim C_k s^{-\frac{1}{2}} \exp\( - \frac{\pi^2}{3k(k+1)s}\).$$
\end{conjecture}
This conjecture proved difficult to establish via standard $q$-series
techniques.  The asymptotic of \cite{HLR} was improved by Mahlburg and the first
author \cite{BM} 
$\log(g_k(e^{-s}))$.  Finally, Kane and the third author 
\cite{KR}, using a technique similar to the transfer matrix method of
statistical mechanics, proved Conjecture \ref{conj:main} with $C_k =
\sqrt{2\pi}/{k}$.

Zagier \cite{zagierPrivate}, using a formula for $g_k$ found by Andrews \cite{AndPNAS}, did extensive 
computations of these asymptotics. He numerically found that, as $s\rightarrow 0$,
$$g_3(e^{-s}) \sim \sqrt{\frac{2\pi}{s}} e^{-\frac{\pi^2}{36s} + \frac{s}{24}  } \( \frac{1}{3} + c_1 s^{\frac{1}{3}}
t_1(s) + c_2 s^{\frac{2}{3}} t_2(s)\),$$
where 
{\small 
\begin{align*}
t_1(s) := &1 -\frac{7}{2^6 3}s -\frac{97}{2^8 3^3}s^2 -\frac{40061}{2^{15} 3^4}s^3 -\frac{18915331}{2^{19} 3^6 5}s^4 -\frac{13796617247}{2^{27} 3^6 5}s^5  - \cdots, \\
t_2(s) :=& 
5  -\frac{29}{2^4 3}s + \frac{19435}{2^{11} 3^3}s^2  -\frac{14885}{2^{12} 3^3} s^3 
+ \frac{51970999}{2^{18} 3^6}s^4 -\frac{28436136277}{2^{24} 3^7 5} s^5 + \cdots,
\end{align*}
}
and 
\begin{equation}\label{constantsZagier}
c_1: = 
 \frac{3^{-\frac{1}{6}} \Gamma\(\frac{1}{3}\) }{8\pi } \  \ \ \text{ and } \ \ \ c_2 := 
 \frac{3^{\frac{1}{6}} \Gamma\(\frac{2}{3} \)}{32 \pi}.
\end{equation}
 The
 computations of Zagier are tantalizing because of the rational values
 appearing in the expansion of $t_1(s)$ and $t_2(s)$ and curious
 because of the powers of $s^{1/3}$ which are atypical in similar
 partition problems. 
Additionally, modular forms arise as generating functions in many partition problems.  
 Knowing that certain generating functions are modular gives one access to deep theoretical tools to prove results in other areas.
On the other hand    proofs of modularity   of $q$-hypergeometric series currently fall
far short of a comprehensive theory to describe the interplay
between them   and automorphic forms.
A recent conjecture of W. Nahm   \cite{Na}
relates the modularity of such series to K-Theory. 
In the situation of interest for this paper \rm with the exception of the case $k=2$, there is no such modular  picture for these generating functions which makes this case much harder. 

We establish Zagier's numerics and its generalizations for all $k$ 

\begin{theorem}\label{thm:main}  
For every $k\in \mathbb{N}$ \text{with} $k>1$, and $N\in \mathbb{N}_0$, we have, as $s\rightarrow 0$,
\begin{align*}
g_k(q) =  \frac{1}{k+1} \sqrt{\frac{2\pi }{s}}e^{-\frac{\pi^2}{3 k(k+1)s} + \frac{s}{24}} \(
\frac{k+1}{k} + \sum_{j=1}^{k N} \beta_{k}(j) s^{\frac{j}{k}} +O\(s^{N}\) 
\),
\end{align*}
where $$\beta_k(j) := b_{k}(j)(k+1)^{-j}k^{\frac{j(k+1)}{k}} 
+ \sum_{kr+\ell = j} b_{k}(\ell)
  \sum_{n=1}^\infty a_{n,r}(-\ell)^n (k+1)^{n-\ell} k^{\frac{\ell(k+1)}{k} - n} $$ 
with 
\begin{equation}\label{bk}
b_{k}(j) :=\frac{k+1}{k\pi j! } (-1)^{j+1} \sin \left( \frac{\pi j(k-1)}{k}\right)\Gamma\left(\frac{j(k+1)}{k}\right)
\end{equation}
and the $a_{n,r}$ are rational numbers defined in \eqref{eqn:a_n_j_def}. 
Moreover, 
 for each $0< j < k$ and $m\in \N$ the values $\beta_k(j+mk)/\beta_k(j) \in \Q$.   
\end{theorem}

\begin{remark}
Theorem \ref{thm:main} confirms Zagier's numerics in the case $k=3$. 
\end{remark}

Our proof technique 
demonstrates the connection between the series $g_k(q)$ and Wright's generalization of the Bessel function 
$$\phi(\rho, \beta; z) := \sum_{n=0}^\infty \frac{z^n}{n! \Gamma(\beta - \rho n)}$$
with $\rho<1$ and $\beta \in \C$. In particular, { as $s\to 0$}, we establish that the leading term in the relative error $R_k(q)$ (Equation \ref{eqn:relative_error_def}) is proportional to the real part of a Wright function 
\[g_k(q) \thicksim  \frac{2}{k+1} \sqrt{\frac{2\pi }{s}}e^{-\frac{\pi^2}{3 k(k+1)s} + \frac{s}{24}}\Re \( \phi \( \frac{k}{k+1},1;\frac{(k+1)^{\frac {k}{k+1}}}{k}e^{\frac{\pi i k}{k+1}}s^{-\frac{1}{k+1}} \) \). \]
In particular, for $k=3,$ a result of Wright {\bf \cite[equation (3.5) and Section 4]{wrightIV} }, gives, { as $s\to 0$}
\begin{equation*}
\frac{1}{2} \Re \( \phi\(\frac{3}{4}, 1; \frac{4^{\frac34}}{3} e^{\frac{3\pi i}{4}} s^{-\frac{1}{4}} \) \)
\sim  \frac{1}{3} + c_1  
s^{\frac{1}{3}} + 5 c_2 
s^{\frac{2}{3}}  + O(s),
\end{equation*}
where $c_1$ and $c_2$ are as in \eqref{constantsZagier}. These are Zagier's asymptotics up to $O(s).$
We believe that such comparison and application of other $q$ analogues of generalized hypergeometric functions may be useful in other asymptotic problems.

The paper is organized as follows.  Section \ref{sec:notation} contains notation and 
basic results about the $q$-functions used throughout the paper.  Section \ref{sec:relative_error}
defines the relative error between the series $g_k$ and the expected main term. 
Section \ref{sec:Wright} shows that the relative error can be approximated by the Wright function 
and its ``moments''.

\section*{Acknowledgements}  
The research of the first author was supported by the Alfried Krupp Prize for Young University Teachers of the Krupp foundation and the research leading to these results has received funding from the European Research Council under the European Union's Seventh Framework Programme (FP/2007-2013) / ERC Grant agreement n. 335220 - AQSER.  The research of the second author was supported by grant project number 27300314 of the Research Grants Council.

The second author thanks Alexandre Eremenko,  Sergei Sitnik, and the MathOverflow users  for helping to name the Wright function.  Part of this work was done while the second author was a graduate student at Drexel University and visiting the University of Cologne. 

 The authors thank Karl Mahlburg for helpful  comments  on an earlier version of this paper.

\section{Notation and Preliminary Results}\label{sec:notation}
This section contains some preliminary results that we require for the proof of Theorem \ref{thm:main} as well as some 
$q$-series notation. 
Wright \cite{wrightII, wrightIII, wrightIV} established asymptotics
for $\phi(\rho, \beta; z)$ in all domains. Unfortunately, a direct application of these asymptotics produces a degenerate answer (\cite[Theorem 1]{wrightIV} with $Y = -\frac{1}{k(k+1)s}$), { namely}  \[ \phi \( \frac{k}{k+1},1;\frac{(k+1)^{\frac {k}{k+1}}}{k}e^{\frac{\pi ik}{k+1}}s^{-\frac{1}{k+1}} \) \thicksim i \sqrt{k(k+1)s}e^{\frac{1}{k(k+1)s}}\sum_{m=0}^{M-1}A_m (-1)^m (k(k+1)s)^{m},   \] where the coefficients $A_m$ are given in \cite{wrightIV}. Taking real parts shows that \[ \Re \( \phi \(\frac{k}{k+1},1;\frac{(k+1)^{\frac {k}{k+1}}}{k}e^{\frac{\pi ik}{k+1}}s^{-\frac{1}{k+1}} \) \) = O(1).\]  A little more nuance needs to be applied to Wright's work to obtain a meaningful estimate.

\begin{proposition}\label{thm:wright_asymptotic}
If $\frac{1}{2} \le \rho < 1$ with $|\arg (-e^{2\pi i \rho})| < \frac \pi 2 (1+ \rho) $ and $z>0,$
then, for $L\in\N $,
$$\Re \( \phi \(\rho, 1; ze^{\pi i\rho} \) \) = \frac{1}{2\rho} + \frac 1{2 \pi \rho} \sum_{\ell =1}^{L-1}  \frac{(-1)^{\ell+1} }{ \ell !}\Gamma\left(\frac{\ell}{\rho}\right)z^{-\frac{\ell}{\rho}} \sin \left( \frac{\pi \ell (2\rho -1)}{\rho}\right)+O\left(z^{-\frac{ L}{\rho}}\right). $$
\end{proposition}
\begin{proof}
We apply the identity \[\frac{1}{\Gamma(z)\Gamma(1-z)} = \frac{1}{\pi} \sin(\pi z)\] and the double angle formula to show that
\begin{align*} \Re \( \phi(\rho, 1; ze^{\pi i\rho}) \) &=\frac{1}{2\rho} + \frac{1}{2\pi } \text{Im}\left( D\left(ze^{2\pi i \rho}\right)\right),
\end{align*} where 
\[  D(w):= \gamma \(\frac{1}{\rho}-1 \) +\frac{1}{\rho} \log (-w)+  \sum_{m=1}^\infty \frac{w^m\Gamma( \rho m)}{m!}.\]
Equation (3.5) of \cite{wrightIV} states that if $|\arg (-w)| < \pi/2(1+ \rho), $ then \[D(w)= \frac{1}{\rho} \sum_{m=1}^{L-1}\frac{(-1)^m}{m!}\Gamma \left(\frac{m}{\rho}\right)(-w)^{-\frac{m}{\rho}} + O\(w^{-\frac{L}{\rho}}\). \]  Note that Wright \cite{wrightIV} used the notation $\sigma =\rho$ and $\beta=1 $. Moreover, our $D(w)$ is $d(w)$ adjusted for the $t=0$ singularity.  This adjustment is discussed Section 4 of the same paper.

\end{proof}
Throughout, we use the following $q$-notation ($z\in \C$): 
\begin{align} 
(z;q)_\infty &:= \prod_{m=0}^\infty (1-zq^m), \nonumber \\
(q;q)_{z} &:= \frac{(q;q)_\infty}{(q^{z+1};q)_\infty}, \label{qsubz}\\
\theta(z,q) &:= \sum_{n\in \mathbb{Z}} (-1)^n z^n q^{n^2}, \nonumber\\
\Gamma_q(z) &:= (q;q)_{z-1}(1-q)^{1-z}.\label{Gammaq}
\end{align}
The Jacobi function has the product expansion (see (100.2) of \cite{rademacher})
\begin{equation}\label{Thetaprod}
  \theta(z, q)=\prod_{n= 1}^\infty\left(1-q^{2n}\right)\left(1-zq^{2n-1}\right) \left(1-z^{-1} q^{2n-1}\right), 
\end{equation}
and satisfies the following inversion formula (with $q:=e^{-s}$ and $z:=e^{2\pi iu})$
(see (38.2) of \cite{rademacher})
\begin{equation}\label{thetatrans} 
 \theta(z, q)=\sqrt{\frac{\pi}{s}}\sum_{n\text{ odd}}e^{-\frac{\pi^2}{4s}(n+2u)^2}.
\end{equation}
Next, we recall two identities due to Euler, which state that [2, equations (2.25) and (2.2.6)]
\begin{align*}
\frac{1}{(z; q)_\infty} &=\sum_{n=0}^\infty \frac{z^n}{(q; q)_n}, \\
(z; q)_\infty &=\sum_{n=0}^\infty \frac{(-1)^n z^n q^{\frac{n(n-1)}{2}}}{(q; q)_n}. 
\end{align*}
Moreover, we require the following asymptotic behavior
\begin{equation}\label{eqn:dedekind_asymptotic}
(q;q)_\infty = \sqrt{2\pi} s^{-\frac12} \exp\left(-\frac{\pi^2}{6s}+\frac{s}{24}\right) \( 1+ O\( e^{-\frac{4\pi^2}{s}}\)\),
\end{equation}
which is easily derived from the transformation formula 
\begin{equation}\label{eqn:eta_transformation}
(q;q)_\infty = \sqrt{\frac{2\pi}{s}} e^{-\frac{\pi^2}{6s} + \frac{s}{24}} \prod_{n=1}^\infty \left(1-e^{-\frac{4\pi^2n}{s}}\right)
\end{equation}
(see (118.5) of \cite{rademacher}).

The following lemma is used in Section \ref{sec:relative_error} to identify terms which can be 
asymptotically ignored in a $q$-hypergeometric expression for $g_k$. 
\begin{lemma}\label{refform}
As $s\to 0$ and $x\to \infty,$ we have \begin{align*}\frac{1}{(q; q)_{x-3}(q; q)_{-x}}&\ll s
  q^{-\frac{x(x-3)}{2}}.
 \end{align*}  \end{lemma}  
\begin{proof} 
By \eqref{Thetaprod}
\[\left(q^{x-2};q\right)_\infty\left(q^{1-x};q\right)_\infty(q;q)_\infty
=\theta\left(q^{x-\frac32},q^\frac{1}{2}\right). \] Dividing by
$(q;q)_\infty^3$ and using (\ref{qsubz}) then results in 
\begin{equation}\label{eqn:lemma_identity}
\frac{1}{(q)_{x-3}(q)_{-x}}
        = \frac{\theta\left(q^{x-\frac32},q^\frac{1}{2}\right)}{(q;q)_\infty^3}.
\end{equation}
By \eqref{eqn:dedekind_asymptotic} 
\[
\frac{1}{(q;q)_\infty^3 } = 
    \frac{s^{\frac32} }{\sqrt{8\pi^3}}\exp\left(\frac{\pi^2}{2s } - \frac{s}{8}\right)\( 1+ O\( e^{-\frac{4\pi^2}{s}}\)\).
\] 
Moreover \eqref{thetatrans} yields
\[
\theta\left(q^{x-\frac32},q^{\frac12}\right) 
 =\sqrt{\frac{8\pi }{s}}\text{Re} \left (\exp\left(\frac{\left(\pi i + s \left(x-\frac{3}{2}\right)\right)^2}{2s}\right)\right)\left(1+O\left(e^{-\frac{4\pi^2}{s}}\right)\right). 
\]
Combining these approximations with \eqref{eqn:lemma_identity} gives 
\begin{align*}\frac{1}{(q;q)_{x-3}(q;q)_{-x}}&=
\frac{s}{\pi}\text{Re}\left (\exp\left(\frac{\left(\pi i + s \left(x-\frac{3}{2}\right)\right)^2+\pi^2}{2s}\right)\right)\left(1+O\left(e^{-\frac{4\pi^2}{s}}\right)\right) \\
&= -\frac{s q^{-\frac{\left(x-\frac{3}{2}\right)^2}{2 }}}{\pi }\sin (\pi x)\left(1+O\left(e^{-\frac{4\pi^2}{s}}\right)\right)=O\left(s q^{-\frac{x(x-3)}{2}}\right).
 \end{align*}
\end{proof}

The following  is derived from   \cite[Theorem 2]{MR1703273} after applying \eqref{eqn:dedekind_asymptotic} 
(see also \cite{zhang}).
\begin{theorem}\label{lem:McIntosh}
For $x\in\R\backslash\{-\N_0\}$, we have
 $N\in \N_0$, and $q = e^{-s}$
\begin{equation*}
\frac{\Gamma(x) }{\Gamma_q(x)} \left(\frac{1-q}{s} \right)^{1-x}q^{\frac{x(x-1)}{2}}
  = q^{\frac{x(x-1)}{4}}\exp \left(-\sum_{j=1}^N \frac{B_{2j} B_{2j+1}(x)}{2j(2j+1)!}s^{2j}  + O_N\left(s^{2N+1}\right)\right), \end{equation*}
where $B_k(x)$ are the Bernoulli polynomials and $B_k$ are the Bernoulli numbers. 
Moreover, 
this asymptotic can be taken to hold on compact subsets of the complex $s$-plane.
\end{theorem}

\section{The Relative Error}\label{sec:relative_error}
In this section, we asymptotically approximate $g_k$ and define a relative error term 
which is then compared to the Wright function. 

We start by representing $g_k$ as an infinite sum of theta functions (see equation (3.3) in \cite{AndPNAS})
\begin{equation}\label{eqn:g_k_with_theta}
g_k(q) =
\frac{1}{\left(q^k;q^k\right)_\infty}\sum_{m=0}^{\infty } \frac{(-1)^m
  q^{\frac{km(m+1)}{2}}\left(q^{k+1-km};q^{k+1}\right)_\infty }{\left(q^k;q^k\right)_m
}\theta\left(q^{km},q^{\frac{k(k+1)}{2}}\right). 
\end{equation}


Turning to the asymptotic expansion of $g_k$, 
it follows from Conjecture \ref{conj:main}, with the constant as established in \cite{KR}, and \eqref{eqn:dedekind_asymptotic} that
\[g_k\left( e^{-s}\right) \thicksim
\frac{k+1}{k} \frac{\left(q^{k+1};q^{k+1}\right)_\infty}{\left(q^k;q^k\right)_\infty}\sqrt{\frac{2\pi
  }{k(k+1)s}}e^{-\frac{\pi^2}{2 k(k+1)s}}.\] 
Thus  it is natural to define the
relative error 
\begin{equation}\label{eqn:relative_error_def}
R_k(q)
:=g_k(q)\frac{\left(q^k;q^k\right)_\infty}{\left(q^{k+1};q^{k+1}\right)_\infty}\sqrt{\frac{k(k+1)s
  }{2\pi }}e^{\frac{\pi^2}{2 k(k+1)s}} 
\end{equation}
 and hence
$\lim_{q\rightarrow 1} R_k (q) = (k+1)/k.$ 

The next lemma transforms the theta term in \eqref{eqn:g_k_with_theta} to identify a 
leading term for the relative error $R_k$ in terms of the $q$-series
\begin{equation*}
 \cI_n(s):= \sum_{m=0}^{\infty}\frac{(-1)^m  e^{ \frac{\pi i mn}{k+1}} q^{\frac{km(m+1)}{2}-\frac{km^2}{2(k+1)} } }{(q^k;q^{k})_m (q^{k+1};q^{k+1})_{-\frac{km}{k+1}}}.
\end{equation*}

\begin{remark}
The function $\cI_1$ is closely related to the 
$q$-Wright function defined in \cite{MR2407280}. The main
difference is that $(k+1)/k$ is not an integer in our case.
\end{remark}
  
\begin{lemma}\label{lem:R_k_summation}  
For every $q\in (0,1),$ we have
\[
R_k(q) = \sum_{n\ \mathrm{odd} }
e^{-\frac{\pi^2 \left(n^2-1\right)}{2 k(k+1)s}} \cI_n(s).
\] 
\end{lemma}

\begin{proof}
Rewriting  \eqref{eqn:g_k_with_theta}, we obtain that
\begin{equation*}\label{eqn:lemma3.1_1}\begin{split}
g_k(q)&=\frac{\left(q^{k+1};q^{k+1}\right)_\infty}{\left(q^k;q^k\right)_\infty}\sum_{m=0}^{\infty}\frac{(-1)^m
  q^{\frac{km(m+1)}{2}}
  \theta\left(q^{km},q^{\frac{k(k+1)}{2}}\right)}{\left(q^k;q^k\right)_m
  \left(q^{k+1};q^{k+1}\right)_{-\frac{km}{k+1}}}. 
	\end{split}
\end{equation*}
Lemma \ref{lem:R_k_summation} now follows by applying the transformation law  \eqref{thetatrans}, to obtain that
 \begin{align*}
  \theta\left(q^{km},q^{\frac{k(k+1)}{2}}\right)  &= \sqrt{\frac{2\pi
    }{ k(k+1)s}} q^{- \frac{km^2}{2(k+1)} }
  \sum_{n\ \textrm{odd}} e^{\frac{\pi im n }{k+1}} e^{- \frac{\pi^2 n^2}{2 k(k+1)s
  }}.
\end{align*}
\end{proof}

The next lemma bounds the terms in the summation for $R_k$
in Lemma \ref{lem:R_k_summation}.
\begin{lemma}\label{lemma3}
For all $n\in \mathbb{N}$ and $s>0,$ we have, as $s\rightarrow 0$,
\[\cI_n(s) =
O\left(\frac{1}{s^3}\exp \left(\frac{\pi^2}{6
  k(k+1)s}\right)\right).\] 
\end{lemma}
\begin{proof} 
Let us first note that for $x>1,$ $(1-q^x)^{-1} < (1-q)^{-1}$, so that
\begin{equation}\label{splitoff}
\left(q;q\right)_x = \frac{\left(q;q\right)_{x+3}}{\left(1-q^{x+3}\right)\left(1-q^{x+2}\right)\left(1-q^{x+1}\right)}
 = O\left(\frac{\left(q;q\right)_{x+3}}{s^3}\right).
\end{equation}
Applying Lemma \ref{refform} with $x=km/\left(k+1\right)$,  yields, using (\ref{splitoff})
\begin{align*}
\frac{1}{\left(q^k;q^{k}\right)_m \left(q^{k+1};q^{k+1}\right)_{-\frac{km}{k+1}}} 
  &=O\left(s \frac{\left(q^{k+1};q^{k+1}\right)_{\frac{km}{k+1}-3}}{\left(q^k;q^k\right)_m}
      q^{-\frac{k^2m^2}{2\left(k+1\right)}+\frac{3 km}{2}}\right) \\
  &=O\left(\frac{\left(q^{k+1};q^{k+1}\right)_{\frac{km}{k+1}}}{s^2 \left(q^k;q^k\right)_m}q^{-\frac{k^2m^2}{2\left(k+1\right)}}\right) \\
&= O\left(\frac{\left(q^{k+1};q^{k+1}\right)_\infty}{\left(q^k;q^k\right)_\infty} \frac{1}{s^2}\frac{\left(q^kq^{km};q^k\right)_\infty}{
 \left(q^{km+k+1};q^{k+1}\right)_\infty}q^{-\frac{k^2m^2}{2\left(k+1\right)}}\right) \\
&=O\left(\frac{\left(q^{k+1};q^{k+1}\right)_\infty}{\left(q^k;q^k\right)_\infty s^2}q^{-\frac{k^2m^2}{2\left(k+1\right)}} \right). 
\end{align*}
The last equality follows since
\[
\left(1-q^{k(m+j)}\right)<\left(1-q^{(k+1)(m+j)}\right) 
\] 
which in particular implies that
\[
\frac{\left(q^kq^{km};q^k\right)_\infty}{\left(q^{km+k+1};q^{k+1}\right)_\infty}<1.
\]
Combining the above gives
\[\cI_n(s) 
= O\left(\frac{(q^{k+1};q^{k+1})_\infty}{(q^k;q^k)_\infty s^2}
    \sum_{m=0}^\infty q^{\frac{k m^2}{2}}q^{-\frac{km^2}{4(k+1)}}  \right).\] 
By bounding the sum against a geometric sum and using \eqref{eqn:eta_transformation}, the claim follows.
\end{proof} 

The next lemma determines the main terms in the summation for $R_k$ in Lemma \ref{lem:R_k_summation} explicitly. 
\begin{lemma}\label{Lemma4}  
For $s>0$ and $N\in\N$, we have 
\[R_k(q)=\cI_1(s) + \cI_{-1}(s) +O\left(s^N\right) 
= 2\mathrm{Re}\left( \cI_1(s)\right) + O\left(s^N\right).\] 
\end{lemma}
\begin{proof}
We have, using Lemma \ref{lem:R_k_summation}, Lemma \ref{lemma3}, 
and the integral comparison test,
\begin{align*}
\abs{R_k(q)-\cI_1(s)-\cI_{-1}(s)}\leq  & \ 
2\frac1{s^3} e^{\frac{\pi^2}{6 k(k+1)s}} \sum\limits_{n\text{ odd}\atop{n\geq 3}}e^{-\frac{\pi^2\left(n^2-1\right)}{2 k(k+1)s}}
 \\
\ll & \ \frac{1}{s^3}e^{\frac{2\pi^2}{3 k (k+1)s}}\int_2^\infty e^{-\frac{\pi^2 x^2}{2 k(k+1)s}}dx
\ll s^N.
\end{align*}
\end{proof}
%

\section{Relative Error in terms of the Wright Function}\label{sec:Wright}

In this section, we continue the study of $\cI_1(s)$, relating it, and thus
the relative error $R_k$, to the Wright function.
 By definition 
 \[\cI_1(s) = \sum_{m=0}^\infty \frac{ e^{\frac{\pi i
      mk}{k+1}} q^{\frac{k^2 m^2}{2(k+1)}
}}{\Gamma_{q^k}(m+1)\Gamma_{q^{k+1}}\left(1-\frac{km}{k+1}\right)} \left(\frac{\left(1-q^{k+1}\right)^{\frac{k}{k+1}}q^{\frac{k}{2}}}{1-q^k}
\right)^m.\] Define $w$ by 
\begin{equation*}
 \frac{(1-q^{k+1})^{\frac{k}{k+1}}q^{\frac{k}{2}}}{1-q^k} \sim  \frac{(k+1)^{\frac{k}{k+1}}}{ks^{\frac{1}{k+1}}} =: w \ \ \text{ as } q\to 1 
\end{equation*}
and write 
\[\cI_1(s) = \sum_{m=0}^\infty \frac{w^m e^{\frac{\pi im k}{k+1}}h_q(m)}{\Gamma(m+1) \Gamma\left(1- \frac{km}{k+1}\right)}, \] 
 where, for $z\in\C$,
\begin{equation*}
h_q(z):= q^{\frac{ k^2 z^2}{2(k+1)}+\frac{kz}{2}}\frac{\Gamma(z+1)\Gamma\left(1-\frac{kz}{k+1}\right)}{\Gamma_{q^k}(z+1)\Gamma_{q^{k+1}}\left(1- \frac{kz}{k+1}\right)} \left(\frac{1-q^{k+1}}{(k+1)s }\right)^{\frac{kz}{k+1}}\left(\frac{ks}{1-q^{k} }\right)^{z}.
\end{equation*} 
\
For every $s>0,$ 
$\Gamma_q(z)$ is, as a function of $z$, a nonzero meromorphic function with
simple poles only if $q^{z+m}=1$ for some $m\in\N_0$. 
Therefore, $\Gamma(z)/\Gamma_q(z)$ can be continued to an entire
function in $z$ and thus {the same is true  for} $h_q(z).$  
Hence, it is possible to define $z$-Taylor 
coefficients  for $h_q(z)$ which converge absolutely and uniformly on
compact subsets of the complex plane.
Namely,
\begin{equation}\label{eqn:hq_expansion}
h_q(z) = a_0(s) + a_1(s)z + a_2(s)z^2 + a_3(s) z^3 + \dots.
\end{equation}
We must then expand each $a_n(s)$ in terms of powers of $s$ and show that while $a_0(s) =1,$  $a_n(s) = O(s)$ as $s\to 0.$
\begin{lemma}\label{lem:a_expansion}
For $N\in\N_0,$ there exists coefficients, such that 
$$a_n(s)=\begin{cases}  1 	&  \text{ if }	n=0, \\
                   a_{n,1}s + a_{n,2} s^2 + \dots + a_{n,N} s^N + O_N\left( s^{N+1}\right) & \text{ if } n>0.
\end{cases}
$$
\end{lemma}
\begin{proof}
First, observe that $h_q(0)=a_0(s)=1$.
Moreover, by definition, we have for $n\in\N_0$
$$ a_n(s) = \int_{0}^{1} h_q\left(e^{2\pi i x}\right) e^{-2\pi i nx} dx.
$$
Theorem \ref{lem:McIntosh} gives that
\begin{align}
\label{eqn:a_n_j_def}  
h_q(z)= \sum_{n=0}^\infty \sum_{j=0}^\infty a_{n,j}s^{j} z^n 
= q^{-\frac{k z^2}{4(k+1)}+  \frac{kz}{2}}\exp \left(- 
\sum_{j=1}^N f_{2j} (z) s^{2j} + O \left( s^{2N+2} \right)
\right),
\end{align} 
where
\[f_{2j} (z) := \frac{B_{2j}\left( B_{2j+1}(1+z) k^{2j} +B_{2j+1}\left(1-\frac{kz}{k+1}\right) (k+1)^{2j} \right)}{2j(2j+1)!}.\]

{To finish the proof, it} remains to be
shown that each $a_n(s)$ has no constant term in its expansion in $s$.
For this, note that the above implies that \[h_q(z) =
1-\frac{k}{2}sz - s \frac{ k^2}{4(k+1)}z^2 + O\left(s^2\right),\]
yielding the claim.
\end{proof}

 Using Lemma
\ref{Lemma4} the relative error becomes
\[R_k(q) = \sum_{j=0}^\infty a_j(s) 2\text{Re}
\left(\phi_j\(\frac{k}{k+1}, 1 ; e^{-\frac{\pi i k}{k+1} } w\)\right)+
O\left(s^N\right), \] where
\begin{equation*}
\phi_j(\rho, \beta; z):=\sum_{m=0}^\infty\frac{m^j z^m}{\Gamma(m+1) \Gamma(\beta - \rho m)}. 
\end{equation*}
Note that $\phi_0(\rho, \beta;z) = \phi(\rho, \beta; z)$ is the usual 
Wright function given in the introduction. 
Define
\begin{align*}
W_j (w) := 2\text{Re}\left( \phi_j\( \frac{k}{k+1}, 1; e^{-\frac{\pi i k}{k+1} } w\)\right).
\end{align*}
In this notation, \eqref{eqn:hq_expansion} and Lemma \ref{lem:a_expansion} yield
\begin{equation}
\label{eqn:Rk_simplified}
R_k(q) = W_0(w) + \sum_{j=1}^\infty a_j(s) W_j(w) + O\left(s^N\right).
\end{equation} 
Since $\frac 12 \leq  k/(k+1)\leq 1 $ and $w\to \infty$ as $s\to 0$ we
are interested in the behavior of the Wright function for $\frac 12 \leq \rho
<1$ as $w\to \infty$. Proposition \ref{thm:wright_asymptotic} applies directly with $\rho = \frac{k}{k+1}$ and yields the following.

\begin{proposition}\label{Theorem6}  For $z>0$ and $L\in\N$, 
\[W_0(z) = \frac{k+1}{k} + \sum_{\ell=1}^{ L}b_k(\ell )z^{- \frac{\ell(k+1)}{k}} + O\left(z^{-\frac{L(k+1)}{k}}\right),\]
where $b_k( \ell)$ is defined in (\ref{bk}). 
\end{proposition}

The following theorem is a slight generalization of Proposition \ref{Theorem6}.
\begin{proposition}\label{thm:W_moment_expansion}
  For every $z>0$ and $j, L\in\N$  \[W_j(z) = \sum_{\ell=1}^{L-1} \left(- \ell \frac{k+1}{k}\right)^j b_k(\ell ) z^{- \frac{\ell(k+1)}{k}} + O\left(z^{- \frac{L(k+1)}{k}}\right). \]
\end{proposition}
\begin{proof} Proposition \ref{Theorem6} gives the asymptotic expansion for  $W_0(z)$. 
Moreover, \[W_j(e^{x}) = \frac{d^j}{dx^j}W_0(e^{x}). \]  Thus,
 for $x\to\infty$, \[W_0(e^x) - \frac{k+1}{k\pi }\sum_{\ell=1}^{
    L}\frac{(-1)^{\ell+1} e^{- \frac{x \ell(k+1)}{k}}\sin\left(
    \frac{\pi \ell(k-1)}{k}\right) \Gamma\left(\frac{\ell(k+1)}{k}\right)}{\ell !}
  -\frac{k+1}{k} \] is an entire function of $x$ which is $O(e^{-
    \frac{Lx(k+1)}{k}}).$ Noting that differentiating still keeps the same big-oh estimate finishes the proof.
\end{proof}
We have now proven what we set out to prove.  

\begin{proof}[Proof of Theorem \ref{thm:main}]
Theorem \ref{thm:main} now follows directly from \eqref{eqn:dedekind_asymptotic}, 
\eqref{eqn:relative_error_def}, \eqref{eqn:Rk_simplified}, and 
Proposition \ref{thm:W_moment_expansion}.
\end{proof}


\begin{thebibliography}{99}

\bibitem{And67a} G. E. Andrews, \emph{Some New Partition Theorems.} J. Comb. Theory {\bf 2}, (1967) 431--436.
\bibitem{And67b} G. E. Andrews, \emph{A Generalization of a Partition Theorem of MacMahon.} J. Comb. Thoery {\bf 3}, 100--101.


\bibitem{AndPNAS} G. E. Andrews, \emph{Partitions with short sequences and mock theta functions}, Proc. Nat. Acad. Sci.
{\bf 102} (2005), 4666--4671.

\bibitem{AEPR} G. E. Andrews, H. Eriksson, F. Petrov, and D. Romik, \emph{Integrals, partitions and MacMahon's theorem.}
J. Comb. Theory (A) {\bf 114} (2007) 545--55.





\bibitem{BM1} K. Bringmann and K. Mahlburg, \emph{An extension of the Hardy-Ramanujan Circle Method
and  applications to partitions without sequences}. American Journal  of Math  {\bf 133}  (2011)  1151--1178.

\bibitem{BM} K. Bringmann and K. Mahlburg, \emph{Improved Bounds on Metastability Thresholds and
Probabilities for Generalized Bootstrap Percolation.} to appear Transactions of the AMS.





\bibitem{MR2407280} M. El-Shahed and A. Salem, \emph{$q$-analogue of Wright function},
    Abstr. Appl. Anal., (2008),  1085--3375.









\bibitem{HLR} A. E. Holroyd, T. M. Liggett, and D. Romik, \emph{Integrals, Partitions, and Cellular Automata.}
Trans. Amer. Math Soc., {\bf 356} (2004)  3349--3368.



\bibitem{KR} D. M. Kane and R. C. Rhoades, 
\emph{A proof of Andrews's conjecture and partitions with no short sequences}, preprint. 


\bibitem{macmahon} P. A. MacMahon, \emph{Combinatorial Analysis.} 
Vol. 2, Cambridge Univ. Press, 1916, reprinted Dover, New York.

\bibitem{MR1703273} R. J. McIntosh,  
  \emph{Some asymptotic formulae for {$q$}-shifted factorials},
   Ramanujan J., {\bf 3}, (1999),  205--214.

\bibitem{Na}
W. Nahm, \emph{Conformal field theory and torsion elements of the Bloch group}, in: Frontiers in number theory, physics, and geometry II, Springer, Berlin, 67-132 (2007).



\bibitem{rademacher} H. Rademacher, \emph{Topics in Analytic Number Theory}, Die Grund. der math. Wiss., Band, 169
(Springer, New York Heidelberg, 1973).



\bibitem{wrightII} E. M. Wright, \emph{The asymptotic expansion of the generalized Bessel function.}
Proc. London Math. Soc. (Ser. II) {\bf 38} (1935), 257--270. 

\bibitem{wrightIII} E. M. Wright, \emph{The asymptotic expansion of the generalized hypergeometric function.}
Journal London Math. Soc. {\bf 10} (1935), 287--293. 

\bibitem{wrightIV} E. M. Wright, \emph{The generalized Bessel function of order greater than one.}
Quart. J. Math., Oxford Ser. {\bf 11} (1940), 36--48. 


\bibitem{zagierDilog} D. Zagier, \emph{The dilogarithm function}, 
In Frontiers in Number Theory, Physics and Geometry II, 
P. Cartier, B. Julia, P. Moussa, P. Vanhove (eds.), Springer-Verlag, Berlin-Heidelberg-New York (2006), 3--65.

\bibitem{zagierPrivate} D. Zagier, \emph{private communication}, (2012). 

\bibitem{zhang} R. Zhang, \emph{On asymptotics of the $q$-exponential and $q$-gamma functions},
  J. Math. Anal. App. {\bf 411} (2014) 522--529. 


\end{thebibliography}
\end{document}